\documentclass{amsart}
\usepackage{amscd,amssymb,amsmath,amsthm}
\usepackage[dvips]{graphicx}
\usepackage{appendix}
\theoremstyle{plain}
\newtheorem{definition}{Definition}
\newtheorem{proposition}{Proposition}

\newtheorem{lemma}[proposition]{Lemma}

\newtheorem*{proposition*}{Proposition}
\newtheorem*{theorem*}{Theorem}
\newtheorem*{corollary*}{Corollary}
\newtheorem*{lemma*}{Lemma}
\newtheorem*{remark*}{Remark}
\newtheorem*{example*}{Example}

\newcommand{\Z}{\mathbb{Z}}
\newcommand{\Q}{\mathbb{Q}}
\newcommand{\R}{\mathbb{R}}

\begin{document}

\title{Open Gromov-Witten theory on Calabi-Yau three-folds II}

\author{Vito Iacovino}

\address{Max-Planck-Institut f\"ur Mathematik, Bonn}

\email{iacovino@mpim-bonn.mpg.de}

\date{version: \today}


\begin{abstract}
We propose a general theory of the Open Gromov-Witten invariant on Calabi-Yau three-folds.
In this paper we construct the Open Gromov-Witten potential. The evaluation of the potential on its critical points leads to numerical invariants.
\end{abstract}

\maketitle

\section{Introduction}
This paper is a continuation of \cite{I1}.
Let $M$ be a Calabi-Yau three-fold and let $L$ be a Special Lagrangian submanifold of $M$. In \cite{I1} we construct the Open Gromov-Witten invariant in the case that each connected component of $L$ has the rational homology of a sphere. In this paper we consider the problem without any restriction on the topology of the Lagrangian.

We construct the Open Gromov-Witten potential $S$, that is the effective action of the Open Topological String (\cite{W}). $S$ is a homotopy class of solutions of the Master equation in the ring of the functions on $H^*(L)$ with coefficients in the Novikov ring.
$S$ is defined up to master homotopy. The master homotopy is unique up to equivalence. 


The construction of $S$ is made in terms of perturbative Chern-Simons integrals on the Lagrangian submanifold. This work makes more apparent the relation with the work of Witten (\cite{W}).

For acyclic connections, the perturbative expansion of Chern-Simons theory has been constructed rigorously by Axelrod and Singer (\cite{AS}) and Kontsevich (\cite{Ko}). The perturbative expansion has been recently generalized to non-acyclic connections by Costello (\cite{Co}) (see also \cite{I}).

We will need to consider only abelian Chern-Simons theory. We will use the geometric approach similar to \cite{I}. The abelian Chern-Simons theory is in some sense trivial since it has not tree-valent vertices. Its partition function is related to the Ray-Singer torsion. The Gromov-Witten potential is defined by some generalization of Wilson loop integral. In the Appendix we recall the geometric construction of the propagator for the abelian Chern-Simons theory.
The usual anomaly of non-abelian Chern-Simons theory (\cite{AS}, \cite{I}) does not enter in the analysis. Therefore it is not necessary to pick a frame of the Lagrangian submanifold. 


In the last section consider the higher genus potential. This is a formal expansion in the string constant $\lambda$. The potential is defined up to quantum master homotopy.

In the particular case that there is an anti-holomorphic involution, our invariant coincide with the one of \cite{So}. It is actually easy to see the contribute of the multi-disks (associated to trees with at least two vertices) cancel out due to the action of the involution. It is also clear that in the higher genus case this cancellation does not hold. This make clear why the argument of \cite{So} cannot be extended to higher genus invariants. It is necessary to consider multi-curves also in this particular case.

In the case of the $S^1$-action considered in \cite{L} the correction of the multi-curves are not zero. Therefore our invariant computes corrections to the invariant of \cite{L}. This correction should be particular relevant in comparison with physics computations of the multi-covering formula and Gupakumar-Vafa invariants.

The evaluation of $S$ on its critical points leads to numerical invariants. In general this invariant is not associated to a fixed relative homology class, but to a fixed area. This problem is already present in the particular case that there is an anti-holomorphic involution (see \cite{So}) where instead of counting disks in a fixed homological class it is possible only to fix the projection of the class in the $-1$ eigenspace of the involution acting on $H_2(M,L,\Q)$.
We believe that the critical points of $S$ are related to the homotopy classes of bounding chains (\cite{FO3}). This correspondence should shed light on the relation between our invariant and the invariant of Joyce (\cite{Jo}).


\section{Systems of homological chains}

For each decorated tree $T$, let $C_T(L)$ be the orbifold
$$ C_T(L) = \left( \prod_{e \in E^(T)} C_e(L) \right) / \text{Aut}(T) .$$
The boundary of $C_T(L)$ can be decomposed in boundary faces corresponding to isomorphims classes of pairs $(T,e)$ where $e$ is an internal edges
$$ \partial_e C_T(L) = \left( C_e(L) \times \prod_{e' \neq e} C_{e'}(L) \right) / \text{Aut}(T,e) . $$

A system of homological chains $W_{\mathcal{T}}$ assigns to each decorated tree $T$ an homological chains
$W_T \in C_{|E(T)|}( C_T(L) , o_T) $
with twisted coefficients in $o_T$. 
We identify two systems of homological chains if they represent the same collection of currents. We assume the following properties.

\begin{itemize}

\item[(B)] For each $T \in \mathcal{T}$, $W_T$ intersects transversely the boundary of $C_T(L)$. For each internal edge $e$ define $ \partial_e W_T = W_T \cap \partial_e C_T(L)$.
Since $\partial C_2(L) \rightarrow L$ is an $S^2$-fibration, the induced map 
\begin{equation} \label{chain-fibration}
\partial_e C_T(L) \rightarrow  L \times C_{T/e}(L) 
\end{equation}
is an $S^2$-fibration.
We assume that there exist homological chains
$$\partial_e' W_T \in C_{|E(T)| -3}( C_{T/e}(L) , o_T )$$ 
such that $\partial_e W_T$ is the geometric preimage of $\partial_e' W_T$ over the $S^2$-fibration (\ref{chain-fibration}). Here we need to consider the homological chains as chains 
\end{itemize}

Let $\partial_e^0 W_T$ be the image of $\partial_e' W_T$ using the projection $ L \times C_{T/e}(L) \rightarrow C_{T/e}(L) $.
Define 
$$ \partial_v W_T =  \sum_{T'/e=(T,v)} \partial_e^0 W_{T'}    $$
where the sum is over all the trees $T'$ and edges $e \in T'$ such that $T'/e \cong (T,v)$.
We assume that  
\begin{equation} \label{boundary-collection}
\partial W_T = \sum_{v \in V(T)} \partial_v W_T + \sum_{e \in E(T)} \partial_e W_T.
\end{equation}
Equation (\ref{boundary-collection}) is considered as an equation of currents.

\subsection{Homotopies}

An homotopy $Y_{\mathcal{T}}$ between $W_{\mathcal{T}}$ and $W_{\mathcal{T}}'$ is a collection of homological chains
$$Y_T \in C_{|E(T)|+1}([0,1] \times C_T(L) , o_T )$$
that satisfy condition $(B)$ and
\begin{equation} \label{boundary-homotopy}
\partial Y_T = \sum_{v \in V(T)} \partial_v Y_T + \sum_{e \in E(T)} \partial_e Y_T + \{ 0 \} \times W_T - \{ 1 \} \times W_T'
\end{equation}

Suppose that $Y_{\mathcal{T}}$ and $X_{\mathcal{T}}$ are two homotopies between $W_{\mathcal{T}}$ and $W_{\mathcal{T}}'$. We say that $Y_{\mathcal{T}}$ is equivalent to $X_{\mathcal{T}}$ if there exists a collection of chains $Z_{\mathcal{T}}$ with
$$Z_T \in C_{|E(T)|+2}([0,1]^2 \times C_T(L) ,o_T )$$
that satisfy condition $(B)$ such that
\begin{eqnarray} \label{boundary-equivalence}
\partial Z_T & \hspace{-0.1in} =& \hspace{-0.1in} \sum_{v \in V(T)} \partial_v Z_T + \sum_{e \in E(T)} \partial_e Z_T \\ 
   & \hspace{-0.1in} +& \hspace{-0.1in} [0,1] \times \{ 0 \} \times W_T - [0,1] \times \{ 1 \} \times W_T' + \{ 0 \} \times Y_T - \{ 1 \} \times X_T
   \nonumber
\end{eqnarray}

\subsection{Degenerate system of chains}

Observe that it is possible to describe a systems of chains as chains on $L^{H(T)}$ instead of $C_T(L)$ as follows. 

For each $T \in \mathcal{T}$ we have an homological chain
$W_T \in C_{|E(T)|}( L^{H(T)}) \otimes o_T$
such that for each $e \in E(T)$ intersect transversely 
$$ \Delta_e= \pi_e^{-1} (\Delta) .$$ 
Here $\Delta \in L^2$ is the diagonal and $\pi_e : L^{H(T)} \rightarrow L^2$ is the natural projection. 
The existence of $W_T$ is equivalent to condition $(B)$.
Define $\partial_e W_T = W_T \cap \Delta_e$. Observe that from $\partial_e W_T$ we can define a chain $\partial_e^0 W_{T}$ on $L^{H(T/e)}$. Equation (\ref{boundary-collection}) is equivalent to
$$ \partial W_T = \sum_{T'/e=T} \partial_e^0 W_{T'} .$$
The same considerations apply to homotopies and equivalence of homotopies.

A germ of of homotopies is given by a collection of chains
$$Y_T \in C_{|E(T)|+1}([0,\delta) \times L^{H(T)}, o_T) $$
for some $\delta >0$ such that on the open interval $(0,\delta)$, $Y_{\mathcal{T}}$ satisfies the compatibility condition on the boundary and the following transversality condition on the intersection $ Y_T \cap (\{ 0 \} \times \Delta_e )$ holds.
Locally we consider $\psi : [0, \delta) \times D \rightarrow [0,\delta) \times L^{H(T)}$ where $D$ is some domain of $\R^{|H(T)|}$. Consider the set $U$ of points $p \in D$ such that $\psi(p,0) \in \pi_e^{-1}(\Delta)$. We assume that over $U$, $ \psi'(p,0)$ defines a section of $ T (L^{H(T)})/ T(\Delta_e) $ that is transverse to the zero section.

A germ of equivalence of homotopies is a collection of chains
$$Z_T \in C_{|E(T)|+2}([0,1] \times [0,\delta) \times L^{H(T)}, o_T ).$$
that are transverse to the diagonals on $[0,1] \times (0, \delta)$, are compatible in the boundary, $ Z_T \cap ([0,1] \times \{ 0 \} \times L^{H(T)}) $ is zero as $|E(T)|+1$ current and the intersection $ Z_T \cap ([0,1] \times \{ 0 \} \times \Delta_e )$ is transverse in the same sense we described above for germs of homotopies.

A degenerate chain is an equivalence class of germs of homotopies. The definition of homotopy and equivalence of homotopies extends naturally to degenerate chains.

\subsection{Gluing property}

Let $\mathcal{T}^1$ be the set of decorated trees with one marked internal edge. Let $\mathcal{T}^{0,1}$ be the set of decorated trees with one marked external edge. Observe that
$$ \mathcal{T}^1 = (\mathcal{T}^{0,1} \times \mathcal{T}^{0,1}) / \Z_2 $$

To the element $(T,e) \in \mathcal{T}^1$ corresponds $(T_1,e_1), (T_2,e_2) \in \mathcal{T}^{0,1}$ where $T_1$ and $T_2$ are the trees made cutting the edge $e$ in two edges $e_1$ and $e_2$.

In this section we will consider system of chains on the set $ \mathcal{T}^1$ and $\mathcal{T}^{0,1}$. The definitions of homotopy of chains and equivalence of homotopy extend straightforwardly to these systems. Also using the forgetful map $\mathcal{T}^1 \rightarrow \mathcal{T}$ and $\mathcal{T}^{0,1} \rightarrow \mathcal{T}$ a system of chain $W_{\mathcal{T}}$ induces a system of chains $W_{\mathcal{T}^1}$ and $W_{\mathcal{T}^{0,1}}$.

\begin{definition} \label{gluing}
A system of chain $W_{\mathcal{T}}$ has the gluing property if
$$ W_{\mathcal{T}^1} = W_{\mathcal{T}^{0,1}} \times W_{\mathcal{T}^{0,1}} .$$
\end{definition}

We wnat now define a notion of gluing property up to homotopy that it is easier to satisfy.

Let $\mathcal{T}^2$ be the set of decorated trees with two marked ordered internal edges. Observe that there exist an action of the $S_2$ on $\mathcal{T}^2$ that switch the order of the marked edges.

Assume that there is an homotopy $Y_{\mathcal{T}^1}$ between $W_{\mathcal{T}^1}$ and $W_{\mathcal{T}^{0,1}} \times W_{\mathcal{T}^{0,1}} $. Let $Y_{\mathcal{T}^2}$ be the induced homotopy on $\mathcal{T}^2$. Then $Y_{\mathcal{T}^2}$ is an homotopy between $W_{\mathcal{T}^2}$ and $W_{\mathcal{T}^{0,1}} \times W_{\mathcal{T}^{1,1}} $ and the composition
\begin{equation} \label{YY-comp}
Y_{\mathcal{T}^2} \circ (W_{\mathcal{T}^{0,1}} \times Y_{\mathcal{T}^{1,1}})
\end{equation}
is an homotopy between $W_{\mathcal{T}^2}$ and $W_{\mathcal{T}^{0,1}} \times W_{\mathcal{T}^{0,2}} \times W_{\mathcal{T}^{0,1}}$.

\begin{definition} \label{gluing-hom}
$W_{\mathcal{T}}$ has the the gluing property up to homotopy if it is assigned an equivalence class of homotopies $Y_{\mathcal{T}^1}$ between $W_{\mathcal{T}^1}$ and $W_{\mathcal{T}^{0,1}} \times W_{\mathcal{T}^{0,1}} $ such that the homotopy (\ref{YY-comp}) is invariant up to equivalence by the switch of the order of the marked edges.
\end{definition}

The following proposition proves that from a system of chains with the gluing property up to homotopy we can construct a system of chains with the gluing property in a canonical way.

\begin{proposition} \label{product-construction}
Let $W_{\mathcal{T}}$ be a system of chains with the gluing property up to homotopy (Definition \ref{gluing-hom}). There exists a system of chain $W_{\mathcal{T}}^0$ with the gluing property and an homotopy $Y_{\mathcal{T}}^0$ between $W_{\mathcal{T}}$ and $W_{\mathcal{T}}^0$ such that
\begin{equation} \label{YY0-comp}
(Y_{\mathcal{T}^{0,1}}^0 \times Y_{\mathcal{T}^{0,1}}^0) \circ Y_{\mathcal{T}^1} \sim Y_{\mathcal{T}^1}^0 
\end{equation}
\end{proposition}

\begin{proof} Let $A \in H_2(M,L)$ and $k \in \Z^+$. Assume that the homotopy $Y^0$ has been constructed on $\mathcal{T}_l(B)$ if $ \omega(B) < \omega (A)$ or $B=A$ and $l < k$. From these data we can define the composition
\begin{equation} \label{YY0-induction}
(Y_{\mathcal{T}^{0,1}}^0 \times Y_{\mathcal{T}^{0,1}}^0)_{\mathcal{T}^1_k(A)} \circ Y_{\mathcal{T}^1_k(A)} .
\end{equation}
Observe that the equivalence class of the image on $\mathcal{T}^2_k(A)$ of (\ref{YY0-induction}) is invariant by the switch of the order of the decorated edges. It follows that there exists an homotopy $Y^0_{\mathcal{T}_k(A)}$ such that the image on $\mathcal{T}^1_k(A)$ is equivalent to (\ref{YY0-induction}).
\end{proof}

Observe that the $W^0_{\mathcal{T}}$ constructed in Proposition \ref{product-construction} is not unique. In each step of the inductive argument we have the freedom to choice a representative of $W_T$ where $T$ is the tree with $k$ external edges and only one vertex in the homology class $A$. It is clear that if $W^0$ and $W^1$ are two system of chains with the gluing property constructed in Proposition \ref{product-construction}, there exists an homotopy with the gluing property between $W^0$ and $W^1$. Moreover the equivalence class of this homotopy is uniquely determined.

\section{Open Gromov-Witten potential}

\subsection{Master Equation}
In the appendix we construct (after choicing some geometric data) the propagator $P \in \Omega^2(C_2(L))$ of the abelian Chern-Simons theory. Observe that the property (\ref{parity}) implies that, for each internal edge $e \in T$, $\pi_e^*(P)$ is a differential two-form with twisted coefficients in $o_e$.



Let $\alpha_i \in \Omega^*(L)$ be a basis of $\Psi$ as in the Appendix. Let $x_i$ be the coordinates on $H^*(L)[1]$ dual on the basis $\alpha_i$.

The differential form
\begin{equation} \label{psi}
\psi = \sum_i x_i \alpha_i \in \mathcal{O}(H^*(L)[1]) \otimes  \Omega^*(L).
\end{equation}
does not depend on the basis $\alpha_i$.

For each $T \in \mathcal{T}$ let $CS_T \in \mathcal{O}(H^*(L)[1]) \otimes  \Omega^*(C_T(L))$ the differential form defined by
\begin{equation} \label{differentialform}
CS_T = \bigwedge_{e \in E^{in}(T)} \pi_e^*(P) \wedge \bigwedge_{e \in E^{ex}(T)} \pi_e^*( \psi).
\end{equation}
Remember that for each $e \in E(T)$ we have an isomorphism $\partial_e C_T(L) \cong \partial C_2(L) \times C_{T/e}$. From (\ref{singularity2}) it follows that  
\begin{equation} \label{singularity3}
i^*_{\partial_e} ( CS_T )= \eta \wedge CS_{T/e} .
\end{equation}
We now prove that (\ref{singularity3}) implies that it is possible to define the integral of the collection of differential forms $CS_{\mathcal{T}}$ on a degenerate chain.

Let $W_{\mathcal{T}}$ be a degenerate chain and let $Y_{\mathcal{T}}$ be a germ of homotopies representing $W_{\mathcal{T}}$.
Define $ Y_T^{\epsilon} = Y_T \cap ( \{ \varepsilon \}  \times C_T(L) ) $. The transversality conditions of degenerate chains imply that $Y_T^{\varepsilon}$ converges as a current on $C_T(L)$. Therefore the limit $ \lim_{\varepsilon \rightarrow 0} \int_{Y_T^{\varepsilon}}  CS_T $ exists.

\begin{lemma} \label{homhom1}
For each $A \in H_2(M,L)$, the limit 
$$ \lim_{\varepsilon \rightarrow 0} \sum_{T \in \mathcal{T}(A)} \int_{Y_T^{\varepsilon}}  CS_T $$
does not depend on the germ of homotopy $Y_{\mathcal{T}}$ representing $W_{\mathcal{T}}$.
\end{lemma}
\begin{proof}

Let $Y_{\mathcal{T}}  $ and $X_{\mathcal{T}} $ be two germs of homotopies representing $W_{\mathcal{T}}$. We need to prove that 
$$ \lim_{\varepsilon \rightarrow 0} \sum_{T \in \mathcal{T}(A)}  \int_{Y_T^{\varepsilon}} CS_T =  \lim_{\varepsilon \rightarrow 0} \sum_{T \in \mathcal{T}(A)}   \int_{X_T^{\varepsilon}} CS_T .$$

Equation (\ref{boundary-equivalence}) implies
$$  \partial Z_T^{\varepsilon} = Y_T^{\varepsilon} - X_T^{\varepsilon} + \sum_{v \in V(T)} \partial_v Z_T^{\varepsilon} + \sum_{e\in E(T)} \partial_e Z_T^{\varepsilon} .$$
By Stokes theorem
\begin{equation} \label{stokes}
\int_{Y_T^{\varepsilon}} CS_T- \int_{X_T^{\varepsilon}} CS_T = \int_{Z_T^{\varepsilon}} d CS_T - \sum_{v \in V(T)} \int_{\partial_v Z_T^{\varepsilon} } CS_T- \sum_{e\in E(T)} \int_{\partial_e Z_T^{\varepsilon}} CS_T.
\end{equation}
We also have
\begin{equation} \label{cancelchain2}
\partial_v Z_T^{\varepsilon} = \sum_{T'/e=(T,v)} \partial_e^0 Z_{T'}^{\varepsilon}
\end{equation}
where the sum over all the tree $T'$ and edges $e \in T'$ such that $T'/e \cong (T,v)$.
Formula (\ref{cancelchain2}) and (\ref{singularity3}) imply that in the sum over all the trees of (\ref{stokes}) the last two terms of (\ref{stokes}) cancel. Therefore
$$ \sum_T  \int_{Y_T^{\varepsilon}}  CS_T - \sum_T \int_{X_T^{\varepsilon}} CS_T = \sum_T \int_{Z_T^{\varepsilon}} d CS_T.$$
The lemma follows since
$$ \lim_{\varepsilon \rightarrow 0} \int_{Z_T^{\varepsilon}} d CS_T =0     $$
because $Z_T^{\varepsilon} \rightarrow 0$ as current on $C_T(L)$.
\end{proof}
Denote by
$$ \int_{W_{\mathcal{T}(A)}} CS_{\mathcal{T}(A)} $$ 
the limit in Lemma \ref{homhom1}.

Suppose now that $W_{\mathcal{T}}$ is a system of chain with the gluing property. 
The effective action (with coefficients in the Novikov ring) is defined by
\begin{equation} \label{action}
S = \sum_{A} S(A) T^{\omega(A)}.
\end{equation}
where $S(A) \in \mathcal{O}(H^*(L))$ is given by
\begin{equation} \label{actionA}
S(A) = \int_{W_{\mathcal{T}(A)}} CS_{\mathcal{T}(A)} 
\end{equation}

Analogously let $Y_{\mathcal{T}}$ be an homotopy of chains with the gluing property.
Let $\pi : Y_{\mathcal{T}} \rightarrow [0,1]$ be the natural projection.
The extended effective action $\tilde{S}$ is defined by
\begin{equation} \label{actionfam}
\tilde{S} = \sum_{A} \tilde{S}(A) T^{\omega(A)}.
\end{equation}
where $\tilde{S}(A)  \in \Omega^*([0,1])  \otimes \mathcal{O}(H^*(L)) $ is given by
\begin{equation} \label{actionAfam}
\tilde{S}(A)=   \pi_* (CS_{\mathcal{T}(A)} ) 
\end{equation}

We have the 
\begin{lemma} \label{master-homotopy}
$\tilde{S}$ is an homotopy of master solutions:
$$d \tilde{S} +\frac{1}{2} \{ \tilde{S}, \tilde{S} \} =0.$$
\end{lemma}
\begin{proof}
The lemma follows directly from the gluing property and formula (\ref{differential})
\end{proof}

We need to consider the dependence of $S$ on the data of that we used to construct to propagator $P$ (see Appendix). Using the argument of \cite{I} it follows that two different data lead to master homotopic solutions. Here the point is to construct the extended propagator for a family of data and use it to define the extended potential as above.  

\subsection{The potential}


Let $W_{\mathcal{T}}$ the system of chain associated to a coherent pertubation in the boundary constructed in \cite{I1}. Recall that two different perturbations lead to homotopic $ W_{\mathcal{T}} $ with the homotopy determined up to equivalence. 

\begin{lemma}
$W_{\mathcal{T}}$ has the gluing property up to homotopy.
\end{lemma}
\begin{proof}

Let $ s_{\mathcal{T}^1} $ and $s_{\mathcal{T}^{0,1}} $ be the pertubations induced by $ s_{\mathcal{T} }$ on $\overline{\mathcal{M}}_{\mathcal{T}^1}(J) $ and $\overline{\mathcal{M}}_{\mathcal{T}^{0,1}}(J) $ respectively. 
We have a natural map
\begin{equation} \label{cutedge3}
\overline{\mathcal{M}}_{\mathcal{T}^1}  \rightarrow ( \overline{\mathcal{M}}_{\mathcal{T}^{0,1}} \times \overline{\mathcal{M}}_{\mathcal{T}^{0,1}} ) / \Z_2
\end{equation}
$ s_{\mathcal{T}^{0,1}} \times s_{\mathcal{T}^{0,1}} $ is a perturbation of the left side of (\ref{cutedge3}). Its pull-back on the right side of (\ref{cutedge3}) is a section of the obstruction bundle of $\overline{\mathcal{M}}_{\mathcal{T}^1} $ that is not transverse in the boundary. There exists a section $\tilde{s}_{\mathcal{T}^1}$ on the obstruction bundle of $\overline{\mathcal{M}}_{\mathcal{T}^1} \times [0,1]$ such that $\tilde{s}_{\mathcal{T}^1}$ 
\begin{itemize}
\item is coherent in the boundary. 
\item restricted to $\overline{\mathcal{M}}_{\mathcal{T}} \times \{ 1 \}$ is equal to $s_{\mathcal{T}^{0,1}} \times s_{\mathcal{T}^{0,1}}$
\item restricted to $\overline{\mathcal{M}}_{\mathcal{T}} \times \{ 0 \}$ is equal to $ s_{\mathcal{T}^1} $
\item is transverse to the zero section outside $\overline{\mathcal{M}}_{\mathcal{T}} \times \{ 0 \}$. 
\end{itemize}
From $\tilde{s}_{\mathcal{T}^1}$ we can construct the homotopy $Y_{\mathcal{T}^1}$ of Definition \ref{gluing-hom}. 
\end{proof}

Let $W_{\mathcal{T}}^0$ be a system of chains constructed in Lemma \ref{product-construction} from the $W_{\mathcal{T}}$ above. As observed after Lemma \ref{product-construction}, two different solutions for $W_{\mathcal{T}}^0$ are connected by an homotopy uniquely determined up to equivalence. We use $W_{\mathcal{T}}^0$ in formula (\ref{action}) to define the Open Gromov-Witten potential $S$.

\begin{proposition}
The Gromov-Witten potential $S$ depends on some choice. To two different choices corresponds a master homotopy (determined up to equivalence) between the associated potentials.
\end{proposition}
\begin{proof}
The Proposition follows from the observations above and Lemma \ref{master-homotopy}.
\end{proof}

\subsection{Enumerative invariants}

An element $\sum_E x_E T^{E}$ is a critical point of $S$ if
\begin{equation} \label{critical}
(\partial_x S)(\sum_E x_E T^{E}) =0
\end{equation}
where the identity (\ref{critical}) has to be expanded as a formal series in $T$.

\begin{lemma} The value of $S$ on its critical points is an invariant.
\end{lemma}
\begin{proof}
Assume that $\tilde{S}$ is a solution of the master homotopy equation as in Lemma \ref{master-homotopy}. Write $ \tilde{S} = S_t + B_t dt  $. The equation $d \tilde{S} + \{ \tilde{S}, \tilde{S} \} =0$ is equivalent to
\begin{equation} \label{master-homotopy2}
\frac{d}{dt} S_t + \{ S_t, B_t \} =0.
\end{equation}
Let $ x_t $ be one parameter smooth family of elements of $\mathcal{O}(H) \otimes \Lambda $ such that $x_t$ is a critical points of $S_t$ for each $t$. By equation (\ref{master-homotopy2}) $ S_t(x_t) =0 $.
Therefore
$$\frac{d}{dt} (S_t(x_t))= (\frac{d}{dt} S_t) (x_t) +  \langle \partial_x S_t , \frac{d}{dt} x_t \rangle =  0.$$
\end{proof}

\section{Higher genus}

We use the same notation of the last section of \cite{I1}.

We can apply the same argument of before to construct a system of chains $W_{\mathcal{G}'}^0$ with the gluing property. Define the effective action $S_{(g,h)}(A) \in \mathcal{O}(H^*(L))$ by
\begin{equation} \label{higher-actionA}
S_{(g,h)}(A) = \int_{W_{\mathcal{G}_{(g,h)}(A)}} CS_{\mathcal{G}} 
\end{equation}

The total action (with coefficients in the Novikov ring and string coupling $\lambda$) is given by
\begin{equation} \label{higher-action}
S = \sum_{A,g,h} \lambda^{2g-2 +h} S_{(g,h)}(A) T^{\omega(A)}.
\end{equation}
The one parameter version $\tilde{S}$ satisfies the quantum master equation
$$d \tilde{S} + \lambda \Delta \tilde{S} + \frac{1}{2} \{ \tilde{S}, \tilde{S} \} =0.$$


\appendix

\section{Abelian Chern-Simons}

Let $C_2(L)$ the geometric blow up along the diagonal of $L^2$ (see \cite{AS}). $C_2(L)$ is a manifold with boundary. The boundary $\partial C_2(L)$ is isomorphic to the sphere normal bundle of the diagonal $\Delta$ of $L \times L$.

We will define the propagator as a differential form of degree two on $C_2(L)$. Our construction depends by the following data
\begin{itemize}
\item a metric on $M$

\item a connection on $TM$ compatible with the metric

\item 
A subspace $\Psi \subset \Omega^*(L)$ of closed differential form such that the natural projection
$$ \Psi \rightarrow   H^*(L)$$
is an isomorphism.
\end{itemize}

Let $\alpha_i \in \Omega(L)$, $\beta_i \in \Omega(L)$ be basis of $\Psi$ such that $\int \alpha_i \wedge \beta_j = \delta_{ij}$.


Define $K \in \Omega^3(L \times L) $ by
\begin{equation} \label{kappa}
K= \sum_i \alpha_i \otimes \beta_i
\end{equation}

The differential form $K$ does not depend by the basis $\alpha_i, \beta_i$.

\subsection{Propagator} Fix an orthogonal frame of $TL$ on an small open subset $U \subset L$. On $U$ the $S(TL)$ is a trivial bundle with fiber $ S^2 $.
Denote by $\theta_i$ the $1$-form components of the connection in this local system. Consider the differential form of the spherical bundle $S(TU)$
\begin{equation} \label{singularity}
\frac{\omega + d(\theta^i x_i)}{ 4 \pi}
\end{equation}
where $\omega$ is the standard volume form of $S^2$ and $x_i$ are the restriction to $S^2$ of the standard coordinates of $\R^3$. The differential form (\ref{singularity}) is independent by the choice of the local frame of $TU$.
It follows that there exists a globally defined differential form $ \eta \in \Omega^2(S(TL))$ such that $\eta$ agree with (\ref{singularity}) for each local frame.

Denote by $\pi_{\partial} : \partial C_2(L) \rightarrow L $ the natural projection. Let $i_{\partial}: \partial C_2(L) \rightarrow C_2(L)$ be the
inclusion.
Let $r : L \times L \rightarrow L \times L$ be the reflection on the diagonal $r(x,y)=(y,x)$. The map $r$ induces a map on $C_2(L)$ that we still denote by $r$.

\begin{lemma}
There exists a differential form $P \in \Omega^2(C_2(M))$ such that
\begin{equation} \label{singularity2}
i_{\partial}^*  P =  \eta
\end{equation}
\begin{equation} \label{differential}
d P = K
\end{equation}
\begin{equation} \label{parity}
r^*  P = - P
\end{equation}
and
\begin{equation} \label{contraction}
\langle P, \alpha_1 \otimes \alpha_2 \rangle =0
\end{equation}
for each $ \alpha_1,\alpha_2 \in \Psi$. 

If $P' \in \Omega^2(C_2(M))$  is another differential form such that (\ref{singularity2}), (\ref{differential}), (\ref{contraction}) and (\ref{parity}) hold, then there exist $\phi \in \Omega^1(C_2(M))$ such that $P-P'= d \phi$ and $i_{\partial}^* \phi=0$.

\end{lemma}

\begin{proof}
Let $U$ be a small tubular neighborhood of the diagonal. Let $\pi_U : U \rightarrow S(TL) $ be the inducted map.
Let $\rho$ be a cutoff function equal to one in a neighborhood of $S(TL)$ and zero outside a compact subset of $U$.
Define preliminarily $ P$ as
$$ P = \rho (\pi_U^* \eta).$$
Equation (\ref{singularity2}) holds.

The differential form $ P$ is closed in a neighborhood of $\partial C_2(TL)$, therefore we can consider $d P$ as a closed form on $\Omega^2(L \times L)$. For any closed differential form $\tau \in \Omega^3(L \times L)$, integrating by parts
we have 
$$ \int_{M^2} (d P) \wedge \tau =  \int_{C_2(L)} (d P) \wedge \tau = \int_{S(TM)} P \wedge i^*_{\Delta} \tau = \int_{\Delta} \tau$$
where in the last equality we have applied (\ref{singularity2}).
It follows that $dP$ and $K$ are in the same cohomology class in $\Omega^3(L \times L )$. Therefore there exists a differential form $\phi \in \Omega^2(L \times L)$ such that
$$ K = dP + d \phi . $$
Replacing $P$ with $P + \phi$ equation (\ref{differential}) holds. 

Observe that (\ref{singularity2}) and (\ref{differential}) do not change if we add to $P$ a closed form of $\Omega^2(L \times L)$. Of course we can find a such differential such that also (\ref{contraction}) holds. 

Finally, $P$ will also satisfy (\ref{parity}) if we choice the cut off function $\rho$ such that $r^* \rho = \rho$ and the differential form that we add to $P$ are antisymmetric.

Now suppose that $P' \in \Omega^2(C_2(M))$ is another differential form as in the lemma.

Since $i_{\partial}^* (P-P')=0$, $P-P'$ defines an element of $H^2(C_2(L),S(TL))$. Since $ H^2(C_2(L),S(TL)) \cong H^2(L \times L, \Delta)$, there exists $\phi_1 \in \Omega^1(C_2(L))$ with $i_{\partial}^* \phi_1 =0$ such that $P-P' -d \phi_1 $ is an element of $\Omega^2(L \times L, \Delta) $. Of course we can assume that $r^*(\phi_1)=-\phi_1$.

Property (\ref{contraction}) and $i_{\partial}^* \phi_1 =0$ imply that $\langle P - P' - d \phi_1, \alpha_1 \otimes \alpha_2 \rangle =0$ for each $\alpha_1, \alpha_2 \in \Psi_2 $.

It follows that the image on $H^2(L \times L)$ of $ P - P' - d \phi_1$ is trivial. Thererefore there exists $\phi_2 \in \Omega^1(L \times L)$ such that $ P - P' - d \phi_1 = d \phi_2$. We can assume that $r^*(\phi_2)=-\phi_2$, and then in particular $i_{\partial}^* (\phi_2) =0$. Therefore
$$  P - P' = d (\phi_1 + \phi_2)  $$
with $i_{\partial}^* (\phi_1 + \phi_2) =0$.




\end{proof}

\end{document}